\documentclass[11pt]{amsart}

\usepackage[utf8]{inputenc}  

\usepackage{amsthm}
\usepackage{amsfonts}
\usepackage{amsmath}
\usepackage{amssymb}
\usepackage{mathtools} 

\usepackage{color}
\definecolor{darkgreen}{rgb}{0,0.45,0}
\definecolor{darkred}{rgb}{0.9,0,0}
\definecolor{darkblue}{rgb}{0,0,0.6}
\usepackage[colorlinks,citecolor=darkgreen,linkcolor=darkred,urlcolor=darkblue]{hyperref}
\usepackage{breakurl}
\usepackage{enumerate} 

\usepackage{tikz}
\usetikzlibrary{arrows}
\usepackage[all]{xy}
\xyoption{2cell}
\xyoption{curve}
\UseTwocells
\input{diagxy}  
\usetikzlibrary{matrix,arrows}

\pgfarrowsdeclare{xyto}{xyto}
{
  \arrowsize=0.22pt
  \advance\arrowsize by .7\pgflinewidth  
  \pgfarrowsleftextend{-10\arrowsize-.5\pgflinewidth}
  \pgfarrowsrightextend{.5\pgflinewidth}
}
{
  \arrowsize=0.22pt
  \advance\arrowsize by .7\pgflinewidth  
  \pgfsetdash{}{0pt} 
  \pgfsetroundjoin	
  \pgfsetroundcap	
  \pgfpathmoveto{\pgfpointorigin}
  \pgfpathcurveto
     {\pgfpoint{-4.2\arrowsize}{0.6\arrowsize}}
     {\pgfpoint{-7.2\arrowsize}{1.5\arrowsize}}
     {\pgfpoint{-10\arrowsize}{3.1\arrowsize}}
  \pgfusepathqstroke
  \pgfpathmoveto{\pgfpointorigin}
  \pgfpathcurveto
     {\pgfpoint{-4.2\arrowsize}{-0.6\arrowsize}}
     {\pgfpoint{-7.2\arrowsize}{-1.5\arrowsize}}
     {\pgfpoint{-10\arrowsize}{-3.1\arrowsize}}
  \pgfusepathqstroke
}

\newlength{\ontotipoffset}

\pgfarrowsdeclare{xyonto}{xyonto}
{
  \arrowsize=0.22pt
  \advance\arrowsize by .7\pgflinewidth  
  \pgfarrowsleftextend{-10\arrowsize-.5\pgflinewidth}
  \pgfarrowsrightextend{.5\pgflinewidth}
}
{
  \arrowsize=0.22pt
  \advance\arrowsize by .7\pgflinewidth  
  \ontotipoffset=-6\arrowsize
  \pgfsetdash{}{0pt} 
  \pgfsetroundjoin	
  \pgfsetroundcap	
  \pgfpathmoveto{\pgfpointorigin}
  \pgfpathcurveto
     {\pgfpoint{-4.2\arrowsize}{0.6\arrowsize}}
     {\pgfpoint{-7.2\arrowsize}{1.5\arrowsize}}
     {\pgfpoint{-10\arrowsize}{3.1\arrowsize}}
  \pgfusepathqstroke
  \pgfpathmoveto{\pgfpointorigin}
  \pgfpathcurveto
     {\pgfpoint{-4.2\arrowsize}{-0.6\arrowsize}}
     {\pgfpoint{-7.2\arrowsize}{-1.5\arrowsize}}
     {\pgfpoint{-10\arrowsize}{-3.1\arrowsize}}
  \pgfusepathqstroke
  \pgfpathmoveto{\pgfpoint{\ontotipoffset}{0}}
  \pgfpathcurveto
     {\pgfpoint{\ontotipoffset-4.2\arrowsize}{0.6\arrowsize}}
     {\pgfpoint{\ontotipoffset-7.2\arrowsize}{1.5\arrowsize}}
     {\pgfpoint{\ontotipoffset-10\arrowsize}{3.1\arrowsize}}
  \pgfusepathqstroke
  \pgfpathmoveto{\pgfpoint{\ontotipoffset}{0}}
  \pgfpathcurveto
     {\pgfpoint{\ontotipoffset-4.2\arrowsize}{-0.6\arrowsize}}
     {\pgfpoint{\ontotipoffset-7.2\arrowsize}{-1.5\arrowsize}}
     {\pgfpoint{\ontotipoffset-10\arrowsize}{-3.1\arrowsize}}
  \pgfusepathqstroke
}

\pgfarrowsdeclare{xytri}{xytri}
{
  \arrowsize=0.22pt
  \advance\arrowsize by .7\pgflinewidth  
  \pgfarrowsleftextend{-.5\pgflinewidth}
  \pgfarrowsrightextend{10\arrowsize+.5\pgflinewidth}
}
{
  \arrowsize=0.22pt
  \advance\arrowsize by .7\pgflinewidth  
  \pgfsetdash{}{0pt} 
  \pgfsetroundjoin	
  \pgfsetroundcap	
  \pgfpathmoveto{\pgfpoint{10\arrowsize}{0}}
  \pgfpathcurveto
     {\pgfpoint{5.8\arrowsize}{0.6\arrowsize}}
     {\pgfpoint{2.8\arrowsize}{1.5\arrowsize}}
     {\pgfpoint{0\arrowsize}{3.1\arrowsize}}
  \pgfpathmoveto{\pgfpoint{10\arrowsize}{0}}
  \pgfpathcurveto
     {\pgfpoint{5.8\arrowsize}{-0.6\arrowsize}}
     {\pgfpoint{2.8\arrowsize}{-1.5\arrowsize}}
     {\pgfpoint{0\arrowsize}{-3.1\arrowsize}}
  \pgflineto{\pgfpoint{0}{3.1\arrowsize}}
  \pgfusepathqstroke
} 
\tikzset{>=xyto}

\newbox\pbbox
\setbox\pbbox=\hbox{\xy \POS(65,0)\ar@{-} (0,0) \ar@{-} (65,65)\endxy}
\def\pb{\save[]+<3mm,-3mm>*{\copy\pbbox} \restore}
\theoremstyle{plain}
\newtheorem{theorem}{Theorem}

\newtheorem{proposition}[theorem]{Proposition}
\newtheorem{lemma}[theorem]{Lemma}
\newtheorem{corollary}[theorem]{Corollary}

\theoremstyle{definition}
\newtheorem{definition}[theorem]{Definition}

\newtheorem{notation}[theorem]{Notation}
\newtheorem{remark}[theorem]{Remark}

\newcommand{\Ebar}{\overline{E}}
\newcommand{\Hbar}{\overline{H}}
\newcommand{\I}{\mathcal{I}}

\renewcommand{\P}{\mathbf{P}}
\newcommand{\PP}{\mathrm{P}}
\newcommand{\Q}{\mathbf{Q}}
\newcommand{\QQ}{\mathrm{Q}}
\newcommand{\U}{\mathbf{U}_{\alpha}}
\newcommand{\UU}{{\mathrm{U}_{\alpha}}}
\newcommand{\UUt}{{\widetilde{\mathrm{U}}_{\alpha}}}
\newcommand{\W}{{\mathbf{W}_{\alpha}}}
\newcommand{\WW}{{\mathrm{W}_{\alpha}}}
\newcommand{\WWt}{{\widetilde{\mathrm{W}}_{\alpha}}}
\newcommand{\y}{\mathbf{y}}

\newcommand{\colim}{\operatorname{colim}}
\newcommand{\Eq}{\mathrm{Eq}}
\newcommand{\Hom}{\operatorname{Hom}}
\newcommand{\op}{\operatorname{op}}
\newcommand{\Sets}{\mathbf{Sets}}
\newcommand{\sSets}{\mathbf{sSets}}

\newcommand{\iso}{\cong}
\newcommand{\adjoint}{\dashv}

\newcommand{\name}[1]{\ulcorner {#1} \urcorner}

\begin{document}
\title{Univalence in simplicial sets}


\author{Chris Kapulkin}
\address[Chris Kapulkin]{University of Pittsburgh}
\email{krk56@pitt.edu}

\author{Peter LeFanu Lumsdaine}
\address[Peter LeFanu Lumsdaine]{Dalhousie University}
\email{p.l.lumsdaine@dal.ca} 

\author{Vladimir Voevodsky}
\address[Vladimir Voevodsky]{Institute for Advanced Study}
\email{vladimir@ias.edu}

\date{March 2012; revised September 2018}

\begin{abstract}
We present an accessible account of Voevodsky’s construction of a univalent universe of Kan fibrations.
\end{abstract}

\maketitle

Our goal in this note is to give a concise, self-contained account of the results of the third-named author given in \cite[Section 5]{voevodsky:notes-on-type-systems}: the construction of a homotopically universal small Kan fibration $\pi \colon \UUt \to \UU$; the proof that $\UU$ is a Kan complex; and the proof that $\pi$ is univalent.

We assume some background knowledge of the homotopy theory of simplicial sets, and category theory in general; \cite{hovey:book} and \cite{mac-lane:cwm} are canonical and sufficient references.  Other good sources include \cite{may:simplicial-book}, \cite{goerss-jardine}, and \cite{joyal:theory-of-quasi-cats}.

In Section~\ref{sec:1}, we construct $\pi \colon \UUt \to \UU$, and prove that it is a weakly universal $\alpha$-small Kan fibration.  In Section~\ref{sec:2}, we prove further that the base $\UU$ is a Kan complex.

Section~\ref{sec:3} is dedicated to constructing the fibration of weak equivalences between two fibrations over a common base.  In Section~\ref{sec:4} we define univalence for a general fibration, and prove our main theorem: that $\pi$ is univalent.  Finally, in Section~\ref{sec:5}, we derive from this a statement of “homotopical uniqueness” for the universal property of $\UU$.

Overall, we largely follow the original presentation of \cite[Section 5]{voevodsky:notes-on-type-systems}, with some departures: some proofs in Sections~\ref{sec:2} and \ref{sec:4} are simplified based on an argument of André Joyal (\cite[Lemma~0.2]{joyal:kan}, cf.~our Lemmas~\ref{lemma:exp_along_cofib},~\ref{Joyal_lemma}); and Section~\ref{sec:3} also is somewhat modified and reorganised.

A recurring theme throughout is that when a map is defined by a “right-handed” universal property, showing that it is a fibration (resp.\ trivial fibration) corresponds to showing that the objects it represents extend along trivial (resp.\ all) cofibrations.

An alternative construction of $\pi \colon \UUt \to \UU$ is sketched in \cite{streicher:univalence}, and an alternative proof of univalence in \cite{moerdijk:univalence}.

This note extracts the purely homotopy-theoretic aspects of \cite{kapulkin-lumsdaine:simplicial-model}; see the introduction of that paper for details of the background of the present work.

\subsection*{Acknowledgements} We particularly thank Michael Warren, for an illuminating series of seminars on the universe construction; Karol Szumi{\l}o, for many helpful conversations, and Mike Shulman, for catching an error in a draft of this note.  The first-named author would like to dedicate this paper to his mother.

The first-named author was financially supported during this work by the NSF, Grant DMS-1001191, and a grant from the Benter Foundation at the University of Pittsburgh; the second-named author, by an AARMS postdoctoral fellowship at Dalhousie University, and grants from NSERC.

\section{Representability of fibrations} \label{sec:1}

\begin{definition}
Let $X$ be a simplicial set. A {\em well-ordered morphism} $f \colon Y \to X$ is a pair consisting of a morphism into $X$ (also denoted by $f$) and a function assigning to each simplex $x \in X_n$ a well-ordering on the fiber $Y_x := f^{-1}(x) \subseteq Y_n$.

If $f \colon Y \to X$, $f' \colon Y' \to X$ are well-ordered morphisms into $X$, an \emph{isomorphism} of well-ordered morphisms from $f$ to $f'$ is an isomorphism $Y \iso Y'$ over $X$ preserving the well-orderings on the fibers.
\end{definition}

\begin{remark}
Since we require no compatibility conditions, there are infinitely many (specifically, $2^{\omega}$) well-orderings even on the map $1 \amalg 1 \to 1$.  The well-orderings are haphazard beasts, and not of intrinsic interest; they are essentially just a technical device to obtain Lemma \ref{lemma:w-preserves-lims}.
\end{remark}

\begin{proposition} \label{prop:well-orderings-rigid}
Given two well-ordered sets, there is at most one isomorphism between them.  Given two well-ordered morphisms over a common base, there is at most one isomorphism between them.
\end{proposition}

\begin{proof}
The first statement is classical, and immediate by induction; the second follows from the first, applied in each fiber.
\end{proof}

\begin{definition}
Fix (once and for all) a regular cardinal $\alpha$.  Say a map $f \colon Y \to X$ is \emph{$\alpha$-small} if each of its fibers $Y_x$ has cardinality $< \alpha$.
\end{definition}

Given a simplicial set $X$ we define $\W (X)$ to be the set of isomorphism classes of $\alpha$-small well-ordered morphisms $p \colon Y \to X$. Given a morphism $f \colon X' \to X$ we define $\W (f) \colon \W (X) \to \W (X')$ by $\W(f)[p] = f^*p$ (the pullback of $p$ along $f$). This gives a contravariant functor $\W \colon \sSets^{\op} \to \Sets$.

\begin{lemma} \label{lemma:w-preserves-lims}
$\W$ preserves limits; i.e., sends colimits in $\sSets$ to limits in $\Sets$.
\end{lemma}

\begin{proof}
Suppose $F \colon \I \to \sSets$ is some diagram, and $X = \colim_\I F$ is its colimit, with canonical co-cone $\nu_i \colon F(i) \to X$.  We need to show that the canonical map $\W(X) \to \lim_{{\I}^{\op}} \W(F(i))$ is an isomorphism.

To see that it is surjective, suppose we are given $[f_i \colon Y_i \to F(i)] \in \lim_\I \W(F(i))$.  For each $x \in X_n$, choose some $i$ and $\bar{x} \in F(i)$ with $\nu(\bar{x}) = x$, and set $Y_x := (Y_i)_{\bar{x}}$.  By Proposition~\ref{prop:well-orderings-rigid}, this is well-defined up to \emph{unique} isomorphism, independent of the choices of representatives $i$, $\bar{x}$, $Y_i$, $f_i$.  The total space of these fibers then defines a well-ordered morphism $f \colon Y \to X$, with fibers smaller than $\alpha$, and with pullbacks isomorphic to $f_i$ as required.

For injectivity, suppose $f, f'$ are well-ordered morphisms over $X$, and $\nu_i^* f \iso \nu_i^* f'$ for each $i$.  By Proposition~\ref{prop:well-orderings-rigid}, these isomorphisms agree on each fiber, so together give an isomorphism $f \iso f'$.
\end{proof}

Define the simplicial set $\WW$ by
\[\WW := \W \circ \y^{\op} \colon \Delta^{\op} \to \Sets,\]
where $\y$ denotes the Yoneda embedding $\Delta \to \sSets$.

\begin{lemma}
The functor $\W$ is representable, represented by $\WW$.
\end{lemma}

\begin{proof}
Given $X \in \sSets$, we have isomorphisms, natural in $X$:
\begin{equation*} \begin{split}
  \W (X) & \cong \W (\colim_{\int X} \Delta [n]) \\
         & \cong \textstyle \lim_{\int X} \W (\Delta[n]) \\
         & \cong \textstyle \lim_{\int X} (\WW)_n \\
         & \cong \textstyle \lim_{\int X} \sSets (\Delta [n], \WW) \\
         & \cong \sSets (\colim_{\int X} \Delta [n], \WW) \\
         & \cong \sSets (X, \WW). \qedhere
\end{split}\end{equation*}
\end{proof}

\begin{notation}
Given an $\alpha$-small well-ordered map $f \colon Y \to X \in \W (X)$, the corresponding map $X \to \WW$ will be denoted by $\name{f}$.
\end{notation}

Applying the natural isomorphism above to the identity map $\WW \to \WW$ gives a universal $\alpha$-small well-ordered simplicial set $\WWt \to \WW$.  Explicitly, $n$-simplices of $\WWt$ are pairs
\[(f \colon Y \to \Delta [n], s \in f^{-1}(1_{[n]}))\]
i.e.\ the fiber of $\WWt$ over an $n$-simplex $\name{f} \in \WW$ is exactly (an isomorphic copy of) the main fiber of $f$.  So, by construction:

\begin{proposition}
The canonical projection $\WWt \to \WW$ is universal for $\alpha$-small well-ordered morphisms.
\end{proposition}

\begin{corollary}
The canonical projection $\WWt \to \WW$ is weakly universal for $\alpha$-small morphisms of simplicial sets; that is, any such morphism can be given (not necessarily uniquely) as a pullback of the projection.
\end{corollary}

\begin{proof}
By the well-ordering principle and the axiom of choice, one can well-order the fibers, and then use the universal property of $\WW$.
\end{proof}

\begin{definition}
 Let $\U \subseteq \W$ (respectively, $\UU \subseteq \WW$) be the subobject consisting of $\alpha$-small well-ordered fibrations\footnote{Here and throughout, by ``fibration'' we always mean ``Kan fibration''.}; and define $\pi \colon \UUt \to \UU$ as the pullback:
 \[\xymatrix{ \UUt \ar[r] \ar[d]_\pi  \pb & \WWt \ar[d] \\
  \UU \ar@{^{(}->}[r] & \WW
 }\]
\end{definition}

We would like to know that $\UU$ is a representing object for fibrations; for this, we must show that fibrationhood is a local condition.

\begin{lemma}\label{U:Kan_fib}
 The map $\pi \colon \UUt \to \UU$ is a fibration.
\end{lemma}

\begin{proof}
 Consider a horn to be filled
  \[\xymatrix{ \Lambda^k [n] \ar[r] \ar@{^{(}->}[d] & \UUt \ar[d]^\pi \\
  \Delta [n] \ar[r]_{\name{x}} & \UU
 }\]
 for some $0 \leq k \leq n$.  It factors through the pullback
   \[\xymatrix{ \Lambda^k [n] \ar[r] \ar@{^{(}->}[d] & \bullet \ar[r] \ar[d]^{x} \pb & \UUt \ar[d]^\pi \\
  \Delta [n] \ar@{=}[r] & \Delta [n] \ar[r]_{\name{x}} & \UU
 }\]
 where by the definition of $\UU$, $x$ is a fibration. Thus the left square admits a diagonal filler, and hence so does the outer rectangle.
\end{proof}

\begin{lemma}
 An $\alpha$-small well-ordered morphism $f \colon Y \to X \in \W (X)$ is a fibration if and only if $\name{f} \colon X \to \WW$ factors through $\UU$.
\end{lemma}

\begin{proof}
 For `$\Rightarrow$', assume that $f \colon Y \to X$ is a fibration. Then the pullback of $f$ to any representable is certainly a fibration:
 \[\xymatrix{ \bullet \ar[r] \ar[d]_{x^*f} \pb & Y \ar[d]^{f} \\
 \Delta [n] \ar[r]_x & X.
 }\]
 so $\name{f}(x) = x^*f \in \UU$, and hence $\name{f}$ factors through $\UU$.

 Conversely, suppose $\name{f}$ factors through $\UU$. Then we obtain:
 \[\xymatrix{ Y \ar[r] \ar[d]_{f} & \UUt \ar[r] \ar[d]^\pi \pb & \WWt \ar[d] \\
 X \ar[r] & \UU \ar@{^{(}->}[r] & \WW,
 }\]
 where the lower composite is $\name{f}$, and the outer rectangle and the right square are pullbacks.  Hence so is the left square, so by Lemma~\ref{U:Kan_fib} $f$ is a fibration.
\end{proof}

As an immediate consequence we obtain the following corollary.

\begin{corollary} \label{cor:U_classifies}
The functor $\U$ is representable, represented by $\UU$.  The map $\pi \colon \UUt \to \UU$ is universal for $\alpha$-small well-ordered fibrations, and weakly universal for $\alpha$-small fibrations.
\end{corollary}

\section{Fibrancy of \protect{$\UU$}}  \label{sec:2}

Our next goal is to prove the following theorem.

\begin{theorem}\label{U:KanCpx}
 The simplicial set $\UU$ is a Kan complex.
\end{theorem}

Before proceeding with the proof we will gather four useful lemmas.  The first two, on the theory of \emph{minimal fibrations}, come originally from \cite{quillen:minimal} and \cite{barratt-gugenheim-moore}.  Since these two lemmas contain all that we need to know about minimal fibrations, we treat the notion as a black box, and refer the interested reader to \cite{may:simplicial-book} for more.

\begin{lemma}[Quillen's Lemma, {\cite{quillen:minimal}}] \label{Quillen_lemma}
Any fibration $f \colon Y \to X$ may be factored as $f = pg$, where $p$ is a minimal fibration and $g$ is a trivial fibration.
\end{lemma}

\begin{lemma}[{\cite[III.5.6]{barratt-gugenheim-moore}}; see also {\cite[Cor.~11.7]{may:simplicial-book}}] \label{May_lemma}
Suppose $X$ is contractible, with $x_0 \in X$, and $p \colon Y \to X$ is a minimal fibration with fiber $F := Y_{x_0}$. Then there is an isomorphism
\[\xymatrix@C=0.5cm{
  Y \ar[rr]^<>(0.5)g \ar[rd]_p & & F \times X \ar[ld]^{\pi_2} \\
  & X &
}\]
over $X$.
\end{lemma}

For the last outstanding lemma, the proof we give is due to André Joyal, somewhat simpler than the original proof in \cite{voevodsky:notes-on-type-systems}. We include details here since the original \cite{joyal:kan} is not currently publicly available.  For this, and again for Theorem~\ref{thm:univalence} below, we make crucial use of exponentiation along cofibrations; so we pause first to establish some facts about this.

\begin{lemma}[Cf.~{\cite[Lemma~0.2]{joyal:kan}}] \label{lemma:exp_along_cofib}
Suppose $i \colon A \to B$ is a cofibration.  Let $i_*$ and $i_!$ denote respectively the right and the left adjoint to the pullback functor $i^* \colon \sSets/B \to \sSets/A$.  Then:
\begin{enumerate}[1.]
\item $i_* \colon \sSets/A \to \sSets/B$ preserves trivial fibrations;
\item $i_* \colon \sSets/A \to \sSets/B$ preserves trivially fibrant objects, i.e.\ sends trivial fibrations into $A$ to trivial fibrations into $B$;
\item the counit $i^* i_* \to 1_{\sSets/A}$ is an isomorphism;
\item if $p \colon E \to A$ is $\alpha$-small, then so is $i_* p$. \end{enumerate}
\end{lemma}

\begin{proof} \
\begin{enumerate}[1.]
\item By adjunction, since $i^*$ preserves cofibrations, $i_*$ preserves trivial fibrations.

\item Just by part 1, plus the fact that as a right adjoint, $i_*$ preserves the terminal object.

\item Since $i$ is mono, $i^* i_! \iso 1_{\sSets/A}$; so by adjointness, $i^*i_* \iso 1_{\sSets/A}$.

\item For any $n$-simplex $x \colon \Delta[n] \to B$, we have $(i_* p)_x \iso \Hom_{\sSets/B}(i^*x,p)$.  As a subobject of $\Delta[n]$, $i^*x$ has only finitely many non-degenerate simplices, so $(i_* p)_x$ injects into a finite product of fibers of $p$ and is thus of size $< \alpha$. \qedhere
\end{enumerate} \end{proof}

\begin{lemma}[{\cite[Lemma~0.2]{joyal:kan}}] \label{Joyal_lemma}
If $t \colon Y \to X$ is a trivial fibration and $j \colon X \to X'$ is a cofibration, then there exists a trivial fibration $t' \colon Y' \to X'$ and a pullback square of the form:
 \[\xymatrix{ Y \ar@{.>}[r] \ar[d]_t \pb & Y' \ar@{.>}[d]^{t'} \\
 X \ar@{^{(}->}[r]_j & X'.
 }\]

If $t$ is $\alpha$-small, then $t'$ may be chosen to also be.
\end{lemma}

\begin{proof}
Take $(Y', t') := j_* (Y, t)$.  By part 2 of Lemma~\ref{lemma:exp_along_cofib}, this is a trivial fibration; by part 2, $j^*Y' \iso Y$; and by part 3, it is small.
\end{proof}

We are now ready to prove that $\UU$ is a Kan complex.

\begin{proof}[Proof of Theorem~\ref{U:KanCpx}]
We need to show that we can extend any horn in $\UU$ to a simplex:
 \[\xymatrix{ \Lambda^k[n] \ar[r] \ar@{^{(}->}[d] & \UU \\
  \Delta [n] \ar@{.>}[ur] & \\
 }\]
By Corollary~\ref{cor:U_classifies}, such a horn corresponds to an $\alpha$-small well-ordered fibration $q \colon Y \to \Lambda^k [n]$.  To extend $\name{q}$ to a simplex, we just need to construct an $\alpha$-small fibration $Y'$ over $\Delta[n]$ which restricts on the horn to $Y$:
\[\xymatrix{
  Y \ar@{.>}[r] \ar[d]_q \pb & Y' \ar@{.>}[d]^{q'} \\
 \Lambda^k[n] \ar@{^{(}->}[r] & \Delta[n].
}\]
By the axiom of choice one can then extend the well-ordering of $q$ to $q'$, so the map $\name{q'} \colon \Delta [n] \to \UU$ gives the desired simplex.

By Quillen's Lemma, we can factor $q$ as
\[\xymatrix{ Y \ar[r]^{q_t} & Y_0 \ar[r]^{q_m} & \Lambda^k[n],}\]
where $q_t$ is a trivial fibration and $q_m$ is a minimal fibration.  Both are still $\alpha$-small: each fiber of $q_t$ is a subset of a fiber of $q$, and since a trivial fibration is onto, each fiber of $q_m$ is a quotient of a fiber of $q$.

By Lemma~\ref{May_lemma}, we have an isomorphism $Y_0 \cong F \times \Lambda^k[n]$ yielding a pullback diagram
 \[\xymatrix{Y_0 \ar@{^{(}->}[r] \ar[d] \pb & F \times \Delta[n] \ar[d] \\
 \Lambda^k[n] \ar@{^{(}->}[r] & \Delta[n]
 }\]
exhibiting $Y_0 \to F \times \Delta^k[n]$ as a cofibration.  By Lemma~\ref{Joyal_lemma}, we can therefore complete the upper square in the following diagram, with both right-hand vertical maps $\alpha$-small fibrations:
 \[\xymatrix{ Y \ar[d]_{q_t} \ar[r] \pb & Y' \ar[d] \\
  Y_0 \ar@{^{(}->}[r] \ar[d]_{q_m} \pb & F \times \Delta[n] \ar[d] \\
 \Lambda^k[n] \ar@{^{(}->}[r] & \Delta[n]
 }.\]

  Since $\alpha$ is regular, the composite of the right-hand side is again $\alpha$-small; so we are done.
\end{proof}

\section{Representability of weak equivalences}  \label{sec:3}

To define univalence, we first need to construct the \emph{object of weak equivalences} between fibrations $p_1 \colon E_1 \to B$ and $p_2 \colon E_2 \to B$ over a common base.  In other words, we want an object representing the functor sending $(X,f) \in \sSets/B$ to the set $\Eq_X(f^*E_1,f^*E_2)$.  As we did for $\U$, we proceed in two steps, first exhibiting it as a subfunctor of a functor more easily seen (or already known) to be representable.

For the remainder of the section, fix fibrations $E_1$, $E_2$ as above over a base $B$. Since $\sSets$ is locally Cartesian closed, we can construct the exponential object between them:

\begin{definition}
 Let $\Hom_B (E_1, E_2) \to B$ denote the internal hom from $E_1$ to $E_2$ in $\sSets/B$.

 Then for any $X$, a map $X \to \Hom_B (E_1,E_2)$ corresponds to a map $f \colon X \to B$, together with a map $u \colon f^*E_1 \to f^*E_2$ over $X$.

 Together with the Yoneda lemma, this implies the explicit description: an $n$-simplex of $\Hom_B (E_1,E_2)$ is a pair
 \[(b \colon \Delta[n] \to B, u \colon b^* E_1 \to b^* E_2) .\]
\end{definition}

\begin{lemma} \label{lemma:HOM_is_fib}
 $\Hom_B (E_1, E_2) \to B$ is a Kan fibration.
\end{lemma}

\begin{proof}
The functor $(-) \times_B E_1 \colon \sSets/B \to \sSets/B$ preserves trivial cofibrations (since $\sSets$ is right proper); so its right adjoint $\Hom_B (E_1,-)$ preserves fibrant objects.
\end{proof}

Within $\Hom_B (E_1, E_2)$, we now want to construct the subobject of weak equivalences.

\begin{lemma} \label{lemma:weqs_pull_back} 
Let $f \colon E_1 \to E_2$ be a weak equivalence over $B$, and suppose $g \colon B' \to B$. Then the induced map between pullbacks $g^*E_1 \to g^*E_2$ is a weak equivalence.
\end{lemma}

\begin{proof}
The pullback functor $g^* \colon \sSets/B \to \sSets/B'$ preserves trivial fibrations; so by Ken Brown's Lemma \cite[Lemma~1.1.12]{hovey:book}, it preserves all weak equivalences between fibrant objects.
\end{proof}

Thus, weak equivalences from $E_1$ to $E_2$ form a subfunctor of the functor of maps from $E_1$ to $E_2$.  To show that this is representable, we need just to show:

\begin{lemma} \label{lemma:weq_fibers}  
 Let $f \colon E_1 \to E_2$ be a morphism over $B$.  If for each simplex $b \colon \Delta[n] \to B$ the induced map $f_b \colon b^*E_1 \to b^* E_2$ is a weak equivalence, then $f$ is a weak equivalence.
\end{lemma}

\begin{proof}
Without loss of generality, $B$ is connected; otherwise, apply the result over each connected component separately.  Take some vertex $b \colon \Delta[0] \to B$, and set $F_i := b^*E_i$.

 Now $\pi_0(f)$ factors as $\pi_0(E_1) \iso \pi_0(F_1) \to^{\pi_0(f_b)} \pi_0(F_2) \iso \pi_0(E_2)$, so is an isomorphism, since by hypothesis $\pi_0(f_b)$ is.  Similarly, for any vertex $e \colon \Delta[0] \to F_1$, we have by the long exact sequence for a fibration:
 \[\mathclap{\xymatrix@C=0.5cm{
 \pi_{n+1} (B, b) \ar[r] \ar[d]_{1} & \pi_{n} (F_1, e) \ar[r] \ar[d]_{\pi_n (f_b)} & \pi_{n} (E_1, e) \ar[r] \ar[d]_{\pi_n(f)} & \pi_{n} (B, b) \ar[r] \ar[d]^{1} & \pi_{n-1} (F_1, e) \ar[d]^{\pi_{n-1} (f_b)} \\
 \pi_{n+1} (B, b) \ar[r] & \pi_{n} (F_2, f(e)) \ar[r] & \pi_{n} (E_2, f(e)) \ar[r] & \pi_{n} (B, b) \ar[r]  & \pi_{n-1} (F_2, f(e))  }}\]
 Each $\pi_n(f_b)$ is an isomorphism, so by the Five Lemma, so is each $\pi_n (f)$.  Thus $f$ is a weak equivalence.
\end{proof}

\begin{definition}
 Let $\Eq_B(E_1, E_2)$ be the simplicial subset of $\Hom_B (E_1, E_2)$ consisting of the $n$-simplices of the form:
 \[ (b \colon \Delta [n] \to B, w \colon b^*E_1 \to b^* E_2)\]
 such that $w$ is a weak equivalence.  (By Lemma~\ref{lemma:weqs_pull_back}, this indeed defines a simplicial subset.)
\end{definition}

From Lemma~\ref{lemma:weq_fibers}, we immediately have:

\begin{corollary}\label{cor:weqs_representable}
 Let $(f, u) \colon X \to \Hom_B (E_1, E_2)$.  Then $u$ is a weak equivalence if and only if $(f, u)$ factors through $\Eq_B (E_1, E_2)$.

 Thus, maps $X \to \Eq_B(E_1,E_2)$ correspond to pairs of maps
 \[(f \colon X \to B, w \colon f^*E_1 \to f^* E_2),\]
 where $w$ is a weak equivalence. \qed
\end{corollary}

While Lemma~\ref{lemma:weq_fibers} was stated just as required by representability, its proof actually gives a slightly stronger statement:

\begin{lemma} \label{lemma:connected_weq}  
 Let $f \colon E_1 \to E_2$ be a morphism over $B$.  If for some vertex $b \colon \Delta [0] \to B$ in each connected component the map of fibers $f_b \colon b^*E_1 \to b^* E_2$ is a weak equivalence, then $f$ is a weak equivalence. \qed
\end{lemma}

\begin{corollary} \label{cor:eq_is_fib}
 The map $\Eq_B(E_1,E_2) \to B$ is a fibration.
\end{corollary}

\begin{proof}
 Suppose we wish to fill a square:
 \[\xymatrix{ \Lambda^k[n] \ar[r] \ar@{^{(}->}[d]^i & \Eq_B(E_1,E_2) \ar[d] \\
  \Delta[n] \ar@{.>}[ur] \ar[r]^b & B \\
 }\]
 By the universal property of $\Eq_B(E_1,E_2)$ this corresponds to showing that we can extend a weak equivalence $w \colon i^*b^*E_1 \to i^*b^*E_2$ over $\Lambda^k[n]$ to a weak equivalence $\overline{w} \colon b^*E_1 \to b^*E_2$ over $\Delta[n]$.

 By Lemma~\ref{lemma:HOM_is_fib}, we can certainly find some map $\overline{w}$ extending $w$.  But then since $\Delta[n]$ is connected, Lemma~\ref{lemma:connected_weq} implies that $\overline{w}$ is a weak equivalence.
\end{proof}

\section{Univalence}  \label{sec:4}

Let $p \colon E \to B$ be a fibration.  We then have two fibrations over $B \times B$, given by pulling back $E$ along the projections.  Call the object of weak equivalences between these $\Eq(E) := \Eq_{B \times B}(\pi_1^*E,\pi_2^*E)$.  Concretely, simplices of $\Eq(E)$ are triples
\[ (b_1,b_2 \in B_n,\, w \colon b_1^*E \to b_2^*E ). \]

By Corollary~\ref{cor:weqs_representable}, a map $f \colon X \to \Eq(E)$ corresponds to a pair of maps $f_1, f_2 \colon X \to B$ together with a weak equivalence $f_1^*E \to f_2^*E$ over $X$.  In particular, there is a diagonal map $\delta \colon B \to \Eq(E)$, corresponding to the triple $(1_B,1_B,1_E)$, and sending a simplex $b \in B_n$ to the triple $(b,b,1_{E_b})$.

There are also source and target maps $s,t \colon \Eq(E) \to B$, given by the composites $\Eq(E) \to B \times B \to^{\pi_i} B$, sending $(b_1,b_2,w)$ to $b_1$ and $b_2$ respectively.  These are both retractions of $\delta$; and by Corollary~\ref{cor:eq_is_fib}, if $B$ is fibrant then they are moreover fibrations.

\begin{definition}
 A fibration $p \colon E \to B$ is called {\em univalent} if $\delta \colon B \to \Eq (E)$ is a weak equivalence.
\end{definition}

Since $\delta$ is always a monomorphism (thanks to its retractions), this is equivalent to saying that $B \to \Eq(E) \to B \times B$ is a (trivial cofibration, fibration) factorisation of the diagonal $\Delta \colon B \to B \times B$, i.e.\ that $\Eq(E)$ is a \emph{path object} for $B$.

\begin{theorem} \label{thm:univalence}
 The fibration $\pi \colon \UUt \to \UU$ is univalent.
\end{theorem}

\begin{proof}
We will show that $t$ is a trivial fibration.  Since it is a retraction of $\delta$, this implies by 2-out-of-3 that $\delta$ is a weak equivalence.

So, we need to fill a square
 \[\xymatrix{
  A \ar[r] \ar@{^{(}->}[d]_i & \Eq(\UUt) \ar[d]^{t} \\
  B \ar[r] \ar@{.>}[ur]     & \UU
 }\]
where $i \colon A\ \to/^{(}->/ B$ is a cofibration.

By the universal properties of $\UU$ and $\Eq(\UUt)$, these data correspond to a weak equivalence $w \colon E_1 \to E_2$ between small well-ordered fibrations over $A$, and an extension $\Ebar_2$ of $E_2$ to a small, well-ordered fibration over $B$; and a filler corresponds to an extension $\Ebar_1$ of $E_1$, together with a weak equivalence $\overline{w}$ extending $w$:
\[\begin{tikzpicture}[x={(1.5cm,0cm)},y={(0cm,1.8cm)},z={(1.8cm,-0.7cm)}]
  \node (A) at (0,0,0) {$A$};
  \node (E1) at (0,1,-0.5) {$E_1$};
  \node (E2) at (0,1,0.5) {$E_2$};
  \draw[->] (E1) to (A);
  \draw[->] (E2) to (A);
  \draw[->,auto] (E1) to node {$\scriptstyle w$} (E2);
  \node (B) at (2,0,0) {$B$};
  \node (Eb1) at (2,1,-0.5) {$\Ebar_1$};
  \node (Eb2) at (2,1,0.5) {$\Ebar_2$};
  \draw[->,dashed] (Eb1) to (B);
  \draw[->] (Eb2) to (B);
  \draw[->,dashed,auto] (Eb1) to node {$\scriptstyle \overline{w}$} (Eb2);
  \draw[->] (A) to (B);
  \draw[->] (E2) to (Eb2);
  \draw (0.2,0.7,0.35) -- (0.4,0.7,0.35) -- (0.4,0.85,0.425); 
  \draw[->,dashed] (E1) to (Eb1);
  \draw[dashed] (0.2,0.8,-0.4) -- (0.4,0.8,-0.4) -- (0.4,0.9,-0.45); 
\end{tikzpicture}\]

As usual, it is sufficient to construct this first without well-orderings on $\Ebar_2$; these can then always be chosen so as to extend those of $E_2$. \\

Recalling Lemmas~\ref{lemma:exp_along_cofib}--\ref{Joyal_lemma}, we define $\Ebar_1$ and $\overline{w}$ as the pullback
 \[\xymatrix{
  \Ebar_1 \ar[d]_{\overline{w}} \ar[r] \pb & i_* E_1 \ar[d]^{i_* w} \\
  \Ebar_2 \ar[r]_{\eta}                    & i_* E_2
 }\]
in $\sSets/B$, where $\eta$ is the unit of $i^* \adjoint i_*$ at $\Ebar_2$.  To see that this construction works, it remains to show:
\begin{enumerate}[(a)]
\item $i^*\Ebar_1 \iso E_1$ in $\sSets/A$, and under this, $i^* \overline{w}$ corrsponds to $w$;
\item $\Ebar_1$ is small over $B$; 
\item $\Ebar_1$ is a fibration over $B$, and $\overline{w}$ is a weak equivalence.
\end{enumerate}

For (a), pull the defining diagram of $\Ebar_1$ back to $\sSets/A$; by Lemma~\ref{lemma:exp_along_cofib} part 3, we get a pullback square
\[\xymatrix{
  i^*\Ebar_1 \ar[d]_{i^*\overline{w}} \ar[r] \pb & E_1 \ar[d]^w \\
  E_2 \ar[r]^{1_{E_2}}                               & E_2
}\]
in $\sSets/A$, giving the desired isomorphism.

For (b), Lemma~\ref{lemma:exp_along_cofib} part 4 gives that $i_*E_1$ is $\alpha$-small over $B$, so $\Ebar_1$ is a subobject of a pullback of $\alpha$-small maps.

For (c), note first that by factoring $w$, we may reduce to the cases where it is either a trivial fibration or a trivial cofibration.

In the former case, by Lemma~\ref{lemma:exp_along_cofib} part 1 $i_*w$ is also a trivial fibration, and hence so is $\overline{w}$; so $\Ebar_1$ is fibrant over $\Ebar_2$, hence over $B$.

In the latter case, $E_1$ is then a deformation retract of $E_2$ over $A$; we will show that $\Ebar_1$ is also a deformation retract of $\Ebar_2$ over $B$.  Let $H \colon E_2 \times \Delta[1] \to E_2$ be a deformation retraction of $E_2$ onto $E_1$.  We want some homotopy $\Hbar \colon \Ebar_2 \times \Delta[1] \to \Ebar_2$ extending $H$ on $E_2 \times \Delta[1]$, $1_{\Ebar_1} \times \Delta[1]$ on $\Ebar_1 \times \Delta[1]$, and $1_{\Ebar_2}$ on $\Ebar_2 \times \{0\}$.  Since these three maps agree on the intersections of their domains, this is exactly an instance of the homotopy lifting extension property, i.e.\ a square-filler
\[\xymatrix{
  (E_2 \times \Delta[1]) \cup (\Ebar_1 \times \Delta[1]) \cup (\Ebar_2 \times \{0\}) 
     \ar@{^{(}->}[d] \ar[rr]^<>(0.4){H \cup 1 \cup 1}   
     & & \Ebar_2 \ar[d] \\
  \Ebar_2 \times \Delta[1] \ar[rr] \ar@{.>}[urr]|{\Hbar} & & B
} \qquad \quad \]
which exists since the left-hand map is a trivial cofibration.  

For $\Hbar$ to be a deformation retraction, we need to see that $\Hbar_{\{1\}} \colon \Ebar_2 \to \Ebar_2$ factors through $\Ebar_1$.  By the definition of $\Ebar_1$, a map $f \colon X \to \Ebar_2$ over $b \colon X \to B$ factors through $\Ebar_1$ just if the pullback $i^*f \colon i^*X \to E_2$ factors through $E_1$.  In the case of  $\Hbar_{\{1\}}$, the pullback is by construction $i^*(\Hbar_{\{1\}}) = (i^*\Hbar)_{\{1\}} = H_{\{1\}} \colon E_2 \to E_2$, which factors through $E_1$ since $H$ was a deformation retraction onto $E_1$.

So $\overline{w}$ embeds $\Ebar_1$ as a deformation retract of $\Ebar_2$ over $B$; thus $\Ebar_1$ is a fibration over $B$ and $\overline{w}$ a weak equivalence, as desired.
\end{proof}

\section{Uniqueness in the universal property of \protect{$\U$}}  \label{sec:5}

Finally, as promised, we will give a uniqueness statement for the representation of a small fibration as a pullback of $\pi \colon \UUt \to \UU$: we show that the space of such representations is contractible.

Let $p \colon E \to B$ be any fibration. We define a functor
\[ \P_{p} \colon \sSets^{\op} \to \Sets \]
taking $\P_{p} (X)$ to be the set of pairs of a map $f \colon X \times B \to \UU$, and a weak equivalence $w \colon X \times E \to f^*\UUt$ over $X \times B$; equivalently, the set of squares
\[\xymatrix{
 X \times E \ar[r]^{f'} \ar[d]_{X \times p} & \UUt \ar[d]^{\pi} \\
 X \times B \ar[r]_{f}                     & \UU
}\]
such that the induced map $X \times E \to f^*\UUt$ is a weak equivalence.  Lemma~\ref{lemma:weqs_pull_back} ensures that this is functorial in $X$, by pullback.

\begin{lemma}
The functor $\P_{p}$ is representable, represented by the simplicial set $(\PP_{p})_n := \P_p (\Delta[n])$.
\end{lemma}

\begin{proof}
Let $\Q_p(X)$ be the set of all commutative squares $(f,f')$ from $p$ to $\UUt \to \UU$; we know that $\Q_p$ is represented by $\QQ_p := E^{\UUt} \times_{E^\UU} B^\UU$.

Now, $\P_p$ is a subfunctor of $\Q_p$.  By Lemma \ref{lemma:weq_fibers}, an element $(f,f') \in \Q_p(X)$ lies in $\P_p(X)$ if and only if for each $x \colon \Delta[n] \to X$, the pullback $x^*(f,f')$ lies in $\P_p(X)$; that is, if its representing map $X \to \QQ_p $ factors through $\PP_p$.
\end{proof}

\begin{proposition}
Let $p$ be an $\alpha$-small fibration. Then $\PP_{p}$ is contractible.
\end{proposition}

\begin{proof}
By Corollary~\ref{cor:U_classifies}, take some map $\name{p} \colon B \to \UU$ such that $E \iso \name{p}^*\UUt$.

Now, for any $X$, maps $X \to \PP_p$ correspond by definition to pairs of maps $f \colon X \times B \to \UU$, $w \colon X \times E \to f^*\UUt$.   But $X \times E \iso (\name{p} \cdot \pi_2)^* \UUt$ over $X$; so such pairs also correspond to maps $\bar{f} \colon X \times B \to \Eq(\UUt)$ such that $s \cdot \bar{f} = \name{p} \cdot \pi_2 \colon X \times B \to \UU$.

From this, we conclude that $\PP_p \to 1$ is a trivial fibration: filling a square
 \[\xymatrix{
  Y \ar[r] \ar[d]       & \PP_p \ar[d] \\
  X \ar[r] \ar@{.>}[ur] & 1
 }\]
corresponds to filling the square
 \[\xymatrix{
  Y \times B \ar[r] \ar[d]                        & \Eq(\UUt) \ar[d]^{s} \\
  X \times B \ar[r]^{\name{p} \cdot \pi_2} \ar@{.>}[ur] & \UU
 }\]
but if $Y \to X$ is a cofibration, then so is $Y \times B \to X \times B$; and by univalence, $s$ is a trivial fibration; so a filler exists.
\end{proof}

\bibliographystyle{amsalphaurlmod}
\bibliography{model-bibliography}

\providecommand{\bysame}{\leavevmode\hbox to3em{\hrulefill}\thinspace}
\providecommand{\MR}{\relax\ifhmode\unskip\space\fi MR }
\providecommand{\MRhref}[2]{%
  \href{http://www.ams.org/mathscinet-getitem?mr=#1}{#2}
}
\providecommand{\href}[2]{#2}
\begin{thebibliography}{BGM59}

\bibitem[BGM59]{barratt-gugenheim-moore}
Michael~G. Barratt, Victor K. A.~M. Gugenheim, and John~C. Moore, \emph{On
  semisimplicial fibre-bundles}, Amer. J. Math. \textbf{81} (1959), 639--657.

\bibitem[GJ09]{goerss-jardine}
Paul~G. Goerss and John~F. Jardine, \emph{Simplicial homotopy theory}, Modern
  Birkh{\"a}user Classics, Birkh{\"a}user, 2009.

\bibitem[Hov99]{hovey:book}
Mark Hovey, \emph{Model categories}, Mathematical Surveys and Monographs,
  vol.~63, American Mathematical Society, Providence, Rhode Island, 1999.

\bibitem[Joy09]{joyal:theory-of-quasi-cats}
Andr{\'e} Joyal, \emph{The theory of quasi-categories and its applications},
  Vol.~II of course notes from Simplicial Methods in Higher Categories, Centra
  de Recerca Matem{\`a}tica, Barcelona, 2008, 2009,
  \url{http://www.crm.es/HigherCategories/notes.html}.

\bibitem[Joy11]{joyal:kan}
Andr\'e Joyal, \emph{A result on {K}an fibrations}, unpublished note, 2011.

\bibitem[KL18]{kapulkin-lumsdaine:simplicial-model}
Chris Kapulkin and Peter~LeFanu Lumsdaine, \emph{The simplicial model of
  univalent foundations (after {V}oevodsky)}, Journal of the European
  Mathematical Society (2018), to appear, \href
  {http://arxiv.org/abs/1211.2851} {\path{arXiv:1211.2851}}.

\bibitem[May67]{may:simplicial-book}
J.~Peter May, \emph{Simplicial objects in algebraic topology}, Mathematical
  Studies, vol.~11, Van Nostrand, 1967.

\bibitem[ML98]{mac-lane:cwm}
Saunders Mac~Lane, \emph{Categories for the working mathematician}, 2 ed.,
  Graduate Texts in Mathematics, vol.~5, Springer-Verlag, New York, 1998.

\bibitem[Moe11]{moerdijk:univalence}
Ieke Moerdijk, \emph{Fiber bundles and univalence}, notes by Chris Kapulkin
  from a talk delivered for Mathematics: Algorithms, and Proofs, December 2011,
  \url{http://www.pitt.edu/~krk56/fiber_bundles_univalence.pdf}.

\bibitem[Qui68]{quillen:minimal}
Daniel~G. Quillen, \emph{The geometric realization of a {K}an fibration is a
  {S}erre fibration}, Proc. Amer. Math. Soc. \textbf{19} (1968), 1499--1500.

\bibitem[Str11]{streicher:univalence}
Thomas Streicher, \emph{A model of type theory in simplicial sets}, unpublished
  note, 2011, \url{http://www.mathematik.tu-darmstadt.de/~streicher/sstt.pdf}.

\bibitem[Voe12]{voevodsky:notes-on-type-systems}
Vladimir Voevodsky, \emph{Notes on type systems}, ongoing unpublished notes,
  2012,
  \url{http://www.math.ias.edu/~vladimir/Site3/Univalent_Foundations_files/expressions_current.pdf}.

\end{thebibliography}

\end{document}